\documentclass[12pt]{article}
\usepackage[T1]{fontenc}
\usepackage[utf8]{inputenc}

\usepackage{amsmath,amsthm,amssymb}     
\usepackage{graphicx}     
\usepackage{hyperref} 
\usepackage{url}
\usepackage{amsfonts} 
\usepackage{geometry}
\usepackage{tikz}
\usepackage{listings}
\usepackage{xcolor}

\usepackage{multicol}

\theoremstyle{definition}
\newtheorem{theorem}{Theorem}
\newtheorem{lemma}[theorem]{Lemma}
\newtheorem{corollary}[theorem]{Corollary}
\newtheorem{definition}[theorem]{Definition}

\newcommand{\genlegendre}[4]{%
  \genfrac{(}{)}{}{#1}{#3}{#4}%
  \if\relax\detokenize{#2}\relax\else_{\!#2}\fi
}
\newcommand{\leg}[3][]{\genlegendre{}{#1}{#2}{#3}}

\allowdisplaybreaks

\makeatletter
\@addtoreset{footnote}{page}
\makeatother

\everymath{\displaystyle}

\begin{document}

\title{Sum of Distinct Biquadratic Residues Modulo Primes}
\author{Samer Seraj\thanks{samer\_seraj@outlook.com, Existsforall Academy}}
\date{\today}
\maketitle

\begin{abstract}
    Two conjectures, posed by Finch-Smith, Harrington, and Wong in a paper published in \textit{Integers} in $2023$, are proven. Given a monic biquadratic polynomial $f(x) = x^4 + cx^2 + e$, we prove a formula for the sum of its distinct outputs modulo any prime $p\ge 7$. Here, $c$ is an integer not divisible by $p$ and $e$ is any integer. The formula splits into eight cases, depending on the remainder of $p$ modulo $8$ and whether $c$ is a quadratic residue modulo $p$. The formula quickly extends to the non-monic case. We then apply the formula to prove a classification of the set of such sums in terms of the sets of squares and fourth powers, when $c$ in $x^4 + cx^2$ is varied over all integers with a fixed prime modulus $p\ge 7$. The sum and the set of sums are manually computed for the excluded prime moduli $p=3,5$.
\end{abstract}

\section{Introduction}

\begin{definition}
For an integer polynomial $f\in \mathbb{Z}[x]$ and a modulus $n$, let
$$R_n (f) = \{f(x) \pmod{n} : x\in\mathbb{Z}\}.$$
be the set of distinct outputs of $f(x)$ modulo $n$ as $x$ ranges over the integers. Note that this is a finite set because the inputs $x=0,1,2,\ldots,n-1$ lead to all outputs modulo $n$, most likely with repetition.
\end{definition}

A few natural questions arise, including the determination of the cardinality $|R_n (f)|$, and the sum $\sum{R_n (f)}$ and product $\prod{R_n (f)}$ of these distinct outputs modulo $n$. The general quadratic sum modulo primes was resolved by Gross, Harrington, and Minott, building on the much earlier work of Stetson \cite{Stetson} on the residues of triangular numbers. The main result of \cite{Harrington2} was to determine the sum of distinct residues module primes of the general cubic polynomial.

Let $p\ge 7$ be any fixed prime modulus for the remainder of the paper, so that the notation $R(f)$ may be used instead of $R_p (f)$. Also, all congruences without a stated modulus are understood to be modulo $p$. The present paper first proves a formula for $\sum{R(x^4 + cx^2 + e)}$ that was conjectured in a paper of Finch-Smith, Harrington, and Wong \cite[p.~7]{Harrington2}. The formula for $c\equiv 0\pmod{p}$ is given in \cite[p.~4]{Harrington2}, so we will focus on $p\nmid c$.

One can also ask for a classification of the structure of the set of distinct sums produced as the coefficients of $f$ vary over the integers. This is defined in a specific relevant case as follows.

\begin{definition}\label{def:set-of-sum-of-residues}
Let
$$S(p) = \left\{\sum{R_p(x^4 + cx^2)} : c = 0,1,2,\ldots,p-1\right\}$$
be the set of all distinct sums $R_p(x^4 + cx^2)$. Note that, since we are working with a polynomial modulo $p$, defining $S(p)$ using $c\in \{0,1,2,\ldots,p-1\}$ is equivalent to using all integers $c$.
\end{definition}

Our second goal is to express the sets of sums $S(p)$ in terms of $R_p (x^2)$ and $R_p (x^4)$, which will resolve another conjecture of Finch-Smith, Harrington, and Wong \cite[p.~6]{Harrington2}.

\begin{definition}\label{def:main-sets}
We define following sets:
\begin{enumerate}
    \item Let $\mathbb{Z}_p = \{0,1,2,\ldots, p-1\}$, with the standard operations of addition and multiplication defined on it modulo $p$.
    \item Let $A_0 = \{a\in\mathbb{Z}_p : \exists x\in\mathbb{Z}_p\backslash\{0\}, x^2 \equiv a \pmod{p}\}$ be the set of quadratic residues modulo $p$, which does not include the $0$ residue. Let $A_{0}' = A_0\cup\{0\}$.
    \item Let $A_1 = \mathbb{Z}_p \backslash A_{0}'$ be the set of quadratic non-residues modulo $p$, which also does not include $0$. Let $A_1 \cup\{0\} = A_{1}'$.
    \item For a set of inputs $S$ and a polynomial $g$, let $R_{a\in S} (g(a))$ denote the set of distinct outputs of $g(a)$ modulo $p$, as $a$ ranges over $S$.
    \item For any positive integer $t$, let $[t]=\{1,2,\ldots,t\}$.
\end{enumerate}
\end{definition}

\begin{definition}\label{def:Legendre-symbol}
The Legendre symbol is defined for integers $a$ as
$$
\leg{a}{p} =
\begin{cases}
    \hfill 1 &\text{ if } a\in A_0\\
    \hfill 0 &\text{ if } p\mid a\\
    \hfill -1 &\text{ if } a\in A_1
\end{cases}.
$$
This lends consistency to the notation of $A_0$ and $A_1$ because $1 = (-1)^0$ and $-1 = (-1)^1$, matching the Legendre symbol.
\end{definition}

\section{The Sum of Distinct Residues}

\begin{lemma}[First supplement to quadratic reciprocity]\label{lem:1st-supplement}
$$
\leg{-1}{p}
=
\begin{cases}
    \hfill 1 &\text{ if } p\equiv 1,5\pmod{8},\\
    \hfill -1 &\text{ if } p \equiv 3,7\pmod{8}
\end{cases}
$$
\end{lemma}

\begin{lemma}[Second supplement to quadratic reciprocity]\label{lem:2nd-supplement}
$$
\leg{2}{p}
=
\begin{cases}
    \hfill 1 &\text{ if } p\equiv 1,7\pmod{8},\\
    \hfill -1 &\text{ if } p \equiv 3,5\pmod{8}
\end{cases}
$$
\end{lemma}

\begin{lemma}\label{lem:closure-cosets}
Let $r \in A_0$ and $s\in A_1$. Due to the fact that the Legendre symbol $\leg{a}{p}$ in the variable $a$ is completely multiplicative, the fact that $|A_0|=|A_1|=\frac{p-1}{2}$, and a standard permutation or coset argument,
\begin{align*}
    r\cdot A_0 &= s\cdot A_1 = A_0,\\
    s\cdot A_0 &= r\cdot A_1 = A_1.
\end{align*}
Here, the multiplication of an element against a set denotes the multiplication of every member of the set by that constant element to produce a new set. As a result, if $p\equiv 3\pmod{4}$, since Lemma \ref{lem:1st-supplement} says $\leg{-1}{p} = -1$, we get $A_1 =(-1)\cdot A_0$, or, equivalently, $(-1)\cdot A_1 = A_0$.
\end{lemma}

We outline several reductions:
\begin{enumerate}
    \item Firstly, since $x^4 + cx^2 + e$ is the composition of $x^2 + cx + e$ over $x^2$, it suffices to instead compute the sum of all distinct $a^2 +ca+e$ modulo $p$, as $a$ ranges over $A_{0}'$. The distinctness of the summands $a^2 +ca+e$ is a key difficulty, as there will likely be repeated summands when $a$ iterates over $A_{0}'$. In terms of the notation described in Definition \ref{def:main-sets}, we denote this set of distinct summands as $R_{a\in A_{0}'} (a^2 + ca + e)$ in order to specify that $a$ must be a quadratic residue or $0$, and not just any arbitrary element of $\mathbb{Z}_p$.
    \item Secondly, since $e$ is a constant,
    $$\sum{R_{a\in A_{0}'}{(a^2 + ca + e)}} \equiv \sum{R_{a\in A_{0}'}{(a^2 + ca)}} + |R_{a\in A_{0}'}{(a^2 + ca)}|\cdot e,$$
    so it suffices to compute the sum and cardinality on the right side separately.
    \item Thirdly, using the notation $t\cdot S$ to mean that all elements of the set $S$ are multiplied by $t$,
    $$R_{a\in A_{0}'}(a^2 + ca) = c^2 \cdot R_{a\in A_{0}'}((c^{-1}a)^2+(c^{-1}a)).$$ By Lemma \ref{lem:closure-cosets},
    $$
    c^{-1}\cdot A_{0}' = 
    \begin{cases}
        \hfill A_{0}' &\text{ if } c\in A_0,\\
        \hfill A_{1}' &\text{ if } c\in A_1
    \end{cases}.
    $$
    If $c\in A_0$, it suffices to compute $\sum{R_{a\in A_{0}'}{(a^2 + ca + e)}}$ as
    $$c^2\cdot\sum{R_{a\in A_{0}'}{(a^2 + a)}} + |R_{a\in A_{0}'}{(a^2 + a)}|\cdot e.$$ If $c\in A_1$, we compute it as
    $$c^2\cdot\sum{R_{a\in A_{1}'}{(a^2 + a)}} + |R_{a\in A_{1}'}{(a^2 + a)}|\cdot e.$$
\end{enumerate}

The basic idea behind computing the sums and cardinalities of the sets $R_{a\in A_{0}'}{(a^2 + a)}$ and $R_{a\in A_{1}'}{(a^2 + a)}$ is to first forget about distinctness of $a^2 + a$ as $a$ ranges over $A_{0}'$ or $A_{1}'$, and then remove the repetition.

\begin{lemma}\label{lem:faulhaber}
For any prime $p\ge 7$ and $k=1,2,$ and $4$,
$$\sum_{i=1}^{p-1}{i^k} \equiv 0\pmod{p}.$$
Subsequently,
$\sum_{a\in A_{0}'}{a^k} \equiv 0 \pmod{p}$ for $k=1$ and $2$.
\end{lemma}

\begin{proof}
Faulhaber's formulas for $k=1,2,$ and $4$ state that
\begin{align*}
    \sum_{i=1}^{p-1}{i} &= \frac{1}{2}\cdot (p-1)p,\\
    \sum_{i=1}^{p-1}{i^2} &= \frac{1}{2\cdot 3}\cdot (p-1)p(2p-1),\\
    \sum_{i=1}^{p-1}{i^4} &= \frac{1}{2\cdot 3\cdot 5}\cdot (p-1)p(2p-1)(3p^2 - 3p -1),
\end{align*}
which are all $0$ modulo $p$ because of the factor $p$. We have used the invertibility of the denominators $2,3,$ and $5$ modulo primes $p\ge 7$ here.

It is well-known that every element $a\in A_{0}$ has exactly two distinct square roots $i\in [p-1]$, and the only square root of $0$ is $0$ which does not alter the sum, so the consequence follows from dividing the second and third zero sums by $2$ or multiplying them by $2^{-1}$.
\end{proof}

As a result of Lemma \ref{lem:faulhaber}, $$\sum_{a\in A_{0}'}(a^2 + a) \equiv 0\pmod{p},$$ which means that what will really matter is the sum of the repeated terms that will be subtracted from this $0$ sum. Moreover, since $0^2 + 0 \equiv 0$ does not contribute anything to a sum,
$$\sum_{a\in A_{0}}(a^2 + a) \equiv \sum_{a\in A_{0}'}(a^2 + a)\pmod{p},$$
and, similarly, the same sums over $A_{1}$ and $A_{1}'$ are congruent.

Now we will classify when $a^2 + a$ is congruent to $b^2 + b$ for $a,b \in A_{0}'$ in order to understand when repetition occurs.

\begin{lemma}\label{lem:overcounting-classification}
Suppose $a\not\equiv b$ and $a^2 + a \equiv b^2 + b$. This holds if and only if $b\equiv -a-1$. Moreover, since we will require both $a,b \in A_{0}'$ or both $a,b\in A_{1}'$, the following criteria will be helpful:
\begin{align*}
1 = \leg{b}{p} &\iff \leg{a+1}{p} = \begin{cases}
    \hfill 1 &\text{ if } p\equiv 1,5\pmod 8,\\
    \hfill -1 &\text{ if } p\equiv 3,7\pmod 8
\end{cases},\\
-1 = \leg{b}{p} &\iff \leg{a+1}{p} =
\begin{cases}
    \hfill -1 &\text{ if } p\equiv 1,5\pmod 8,\\
    \hfill 1 &\text{ if } p\equiv 3,7\pmod 8
\end{cases}
\end{align*}
Also, $a^2 + a\equiv 0$ if and only if $a\equiv 0,-1$.
\end{lemma}

\begin{proof}
Assuming $a\not\equiv b$, we manipulate $$a^2 +a\equiv b^2 +b \iff (a-b)(a+b+1) \equiv 0 \iff b\equiv -a-1.$$
For the second point,
$$(-1)^i = \leg{b}{p} = \leg{-a-1}{p} = \leg{-1}{p} \leg{a+1}{p},$$\
which is true if and only if
$\leg{a+1}{p} = (-1)^i \cdot \leg{-1}{p},$
where $i=0,1$. The result follows from applying Lemma \ref{lem:1st-supplement}.

Lastly, $0\equiv a^2 + a \equiv a(a+1)$ if and only if $a\equiv 0,-1$. Note that, by Lemma \ref{lem:1st-supplement}, $-1$ is not always a quadratic residue modulo $p$.
\end{proof}

Lemma \ref{lem:overcounting-classification} motivates the definitions of the following sets.

\begin{definition}
Following the notation in \cite[p.~29]{Lemmermeyer}, for $i,j\in\{0,1\}$, let
$$A_{ij} = \left\{a\in [p-1] : a\in A_i, a+1\in A_j\right\}.$$
For example, $A_{01}$ consists of $a\in [p-1]$ such that $a$ is a quadratic residue and $a+1$ is a quadratic non-residue. In fact, we could have used $[p-2]$ instead of $[p-1]$ because, for $a\equiv p-1$, we have $a+1 \equiv 0$, which is not included in $A_{0}$ or $A_{1}$.
\end{definition}

It is evident from Lemma \ref{lem:overcounting-classification} that we must study the structure of the $A_{ij}$. Next, Lemma \ref{lem:A00-generation} studies a construction of $A_{00}$.

\begin{lemma}\label{lem:A00-generation}
Modulo $p$, the set $A_{00}$ is ``generated'' by the list $\left[\left(\frac{w^{-1}-w}{2}\right)^2\right]_{w=1}^{p-1}$ in the sense that:
\begin{enumerate}
    \item Every element of $A_{00}$ appears at least once.
    \item Elements of $\mathbb{Z}_p$ outside of $A_{00}\cup\{0\}$ do not appear, except $-1$, which appears if $-1 \in A_0$, and $0$, which always appears.
    \item Every element of $A_{00}$, which is a set that excludes $0$ and $-1$, appears exactly four times.
\end{enumerate}
Since $0^2 + 0 \equiv (-1)^2 + (-1) \equiv 0$, their irregular number of occurrences does not affect the fact that
$$\sum_{a\in A_{00}}{(a^2 +a)} \equiv \frac{1}{4}\cdot \sum_{w=1}^{p-1}{\left[\left(\frac{w^{-1}-w}{2}\right)^4+\left(\frac{w^{-1}-w}{2}\right)^2\right]}.$$
\end{lemma}

\begin{proof}
Suppose $(x,y)\in\mathbb{Z}_{p}^{2}$ such that $x^2 + 1 \equiv y^2$. Then
$$1\equiv y^2 - x^2 \equiv (y-x)(y+x).$$
Since $y\not\equiv x$ (otherwise the above congruence would equal $0$ instead of $1$), let $w\equiv y-x$, and so $y+x \equiv w^{-1}$. Solving this system yields
$$(x,y) = \left(\frac{w^{-1}-w}{2},\frac{w^{-1}+w}{2}\right),$$ which shows that $x^2$ does appear in the stated list. It may be verified that, for these $(x,y)$, it holds that $x^2 + 1 = y^2$. So, only elements of $A_{00}\cup\{0\}$, along with $-1$ if $-1\in A_0$, are generated by the list.

Now we count repetition in the list. Suppose $w,z\in\mathbb{Z}_p\backslash\{0\}$ such that
\begin{align*}
    &\left(\frac{w^{-1}-w}{2}\right)^2 \equiv \left(\frac{z^{-1}-z}{2}\right)^2\\
    \iff& 0 \equiv (w^{-1}-w)^2 - (z^{-1}-z)^2\\
    \iff & 0 \equiv (w-z)(w+z)(wz-1)(wz+1)\iff z\equiv \pm w, \pm w^{-1}.
\end{align*}
We claim that some of these four $z$ are congruent if and only if $w^2 \equiv \pm 1$. It is impossible that $w\equiv -w$ or $w^{-1}\equiv -w^{-1}$ because $p\ne 2$ and $p\nmid w$. If $w\equiv w^{-1}$ or $w\equiv w^{-1}$, then $w^2 \equiv 1$. If $w\equiv -w^{-1}$ or $-w\equiv w^{-1}$, then $w^2 \equiv -1$. In these cases, $$\left(\frac{w^{-1}-w}{2}\right)^2 \equiv \frac{w^{-2} + w^{2}-2}{4} \equiv 0,-1.$$ Thus, all congruence classes of $A_{00}$ are covered exactly four times. Note that $w^2 \equiv -1$ occurs for some $w\in \mathbb{Z}_p$ precisely when $p\equiv 1,5\pmod{8}$, by Lemma \ref{lem:1st-supplement}.
\end{proof}

There are times when $a=\frac{p-1}{2}$ appears in $A_{ij}$. In such cases, it requires special consideration because it is the only $a\in\mathbb{Z}_p$ for which $(a,a+1)$ and $(b,b+1)=(-a-1,-a)$ are identical. This leads to studying Lemma \ref{lem:plusminus-2-legendre} below.

\begin{lemma}\label{lem:plusminus-2-legendre}
Note that $\left(\frac{\left(\frac{p-1}{2}\right)}{p}\right)=\leg{-2}{p}=\leg{-1}{p}\leg{2}{p}$ and $\left(\frac{\left(\frac{p+1}{2}\right)}{p}\right)=\leg{2}{p}$. As a direct consequence of Lemma \ref{lem:1st-supplement} and Lemma \ref{lem:2nd-supplement},
\begin{align*}
    \left(\leg{-2}{p},\leg{2}{p}\right) =
    \begin{cases}
        \hfill (1,1) &\text{ if } p\equiv 1\pmod{8}\\
        \hfill (1,-1) &\text{ if } p\equiv 3\pmod{8}\\
        \hfill (-1,-1) &\text{ if } p\equiv 5\pmod{8}\\
        \hfill (-1,1) &\text{ if } p\equiv 7\pmod{8}
    \end{cases}.
\end{align*}
\end{lemma}

We will now compute the values of the subtraction sums modulo $p$.

\begin{lemma}\label{lem:subtraction-sum-computations}
Modulo $p$, we compute
\begin{align*}
    \sum_{a\in A_{00}}{(a^2+a)} & \equiv \sum_{a\in A_{11}}{(a^2+a)}\equiv \frac{1}{32},\\
    \sum_{a\in A_{01}}{(a^2+a)} & \equiv \sum_{a\in A_{10}}{(a^2+a)}\equiv -\frac{1}{32}.
\end{align*}
\end{lemma}

\begin{proof}
The computations begin with noticing the following splittings, where it is used that excluding the index $a=-1$, which is sometimes in $A_0$ or $A_1$ on the left, from the sums on the right (because $a+1\equiv 0$) does not make a difference to the identities,
\begin{align*}
    \sum_{a\in A_{0}}{(a^2 +a)} &\equiv \sum_{a\in A_{00}}{(a^2 +a)} + \sum_{a\in A_{01}}{(a^2 +a)},\\
    \sum_{a\in A_{1}}{(a^2 +a)} &\equiv \sum_{a\in A_{10}}{(a^2 +a)} + \sum_{a\in A_{11}}{(a^2 +a)}.
\end{align*}
Moreover, using Lemma \ref{lem:faulhaber},
$$\sum_{a\in A_{0}}{(a^2 +a)} + \sum_{a\in A_{1}}{(a^2 +a)} \equiv \sum_{a=1}^{p-1}{(a^2+a)} \equiv 0.$$
Next, we use Lemma \ref{lem:A00-generation} to compute
\begin{align*}
    \sum_{a\in A_{00}}{(a^2 +a)} &\equiv \frac{1}{4}\cdot \sum_{w=1}^{p-1}{\left[\left(\frac{w^{-1}-w}{2}\right)^4+\left(\frac{w^{-1}-w}{2}\right)^2\right]}\\
    &\equiv \frac{1}{4}\cdot \sum_{w=1}^{p-1}{\left(\frac{w^{-4}+w^4 -2}{16}\right)} \equiv \frac{1}{4}\cdot\sum_{w=1}^{p-1}{\frac{w^4-1}{8}}\\
    &\equiv \frac{1}{32}\cdot \left(\sum_{w=1}^{p-1}{w^4}-\sum_{w=1}^{p-1}{1}\right) \equiv -\frac{p-1}{32} \equiv \frac{1}{32},
\end{align*}
where Lemma \ref{lem:faulhaber} was used again. We remarked after Lemma \ref{lem:faulhaber} that $$\sum_{a\in A_{0}}{(a^2 +a)} \equiv \sum_{a\in A_{0}'}{(a^2 +a)} \equiv 0,$$ which leads to
$$\sum_{a\in A_{01}}{(a^2 +a)} \equiv \sum_{a\in A_{0}}{(a^2 +a)}- \sum_{a\in A_{00}}{(a^2 +a)} \equiv -\frac{1}{32}.$$

By using changes of variables and Lemma \ref{lem:faulhaber} again,
\begin{align*}
    \sum_{a\in A_{10}}{(a^2 +a)} &\equiv \sum_{\substack{a\in A_{1} \\ (a+1)\in A_{0}}}{a(a+1)}\equiv \sum_{\substack{(a-1)\in A_{1} \\ a\in A_{0}}}{(a-1)a}\\
    &\equiv \sum_{a\in A_{0}}{a(a-1)} - \sum_{\substack{(a-1)\in A_{0} \\ a\in A_{0}}}{a(a-1)}\\
    &\equiv \frac{1}{2}\cdot\sum_{i=1}^{p-1}{(i^2 -i)} - \sum_{\substack{a\in A_{0} \\ (a+1)\in A_{0}}}{(a+1)a}\\
    &\equiv 0-\sum_{a\in A_{00}}{(a^2 +a)} \equiv -\frac{1}{32}.
\end{align*}

In the second line of the computations above, there should technically have been a subtraction of $1(1-1) \equiv 0$ to account for the index $a-1=0$, but this is of no consequence. For the final computation, we see that
\begin{align*}
    \sum_{a\in A_{10}}{(a^2 +a)} + \sum_{a\in A_{11}}{(a^2 +a)} &\equiv \sum_{a\in A_{1} }{(a^2 +a)}\equiv -\sum_{a\in A_{0}}{(a^2 +a)} \equiv 0,\\
    \sum_{a\in A_{11}}{(a^2 +a)} &\equiv -\sum_{a\in A_{10}}{(a^2 +a)}\equiv \frac{1}{32}.
\end{align*}
\end{proof}

The legwork so far will be able to handle the sum part of the computation. For the cardinality part, we quote Lemma \ref{lem:Lemmermeyer} below.

\begin{lemma}\label{lem:Lemmermeyer}
According to \cite[p.~29]{Lemmermeyer},
\begin{enumerate}
    \item If $p\equiv 1,5\pmod{8}$, then $$|A_{00}| = \frac{p-5}{4},|A_{01}| = |A_{10}| = |A_{11}| = \frac{p-1}{4}.$$
    \item If $p\equiv 3,7\pmod{8}$, then $$|A_{01}| = \frac{p+1}{4},|A_{00}| = |A_{10}| = |A_{11}| = \frac{p-3}{4}.$$
\end{enumerate}
\end{lemma}

Now, we are ready to prove the theorem that computes the sum.

\begin{theorem}\label{thm:master-theorem}
Let $p\ge 7$ be a prime. Then
$$\sum{R(x^4 + cx^2 + e)} \equiv Vc^2 + We \pmod{p},$$ where
we compute $V$ and $W$, using the collected data in the table beneath, as
\begin{align*}
    V&\equiv \begin{cases}
        \hfill \sum{R_{a\in A_{0}'}(a^2 +a)} &\text{ if } \leg{c}{p} = 1\\
        \hfill \sum{R_{a\in A_{1}'}(a^2 +a)} &\text{ if } \leg{c}{p} = -1\\
    \end{cases}\\
    &\equiv -\frac{1}{2}\cdot \left[\sum_{a\in A_{ij}}{(a^2 + a)}-\left(-\frac{1}{4} \text{ if ``Yes''}\right)\right] \pmod{p},
\end{align*}
and
\begin{align*}
    W&\equiv \begin{cases}
        \hfill |R_{a\in A_{0}'}(a^2 +a)| &\text{ if } \leg{c}{p} = 1\\
        \hfill |R_{a\in A_{1}'}(a^2 +a)| &\text{ if } \leg{c}{p} = -1\\
    \end{cases}\\
    &\equiv |A_i\backslash\{-1\}| - \frac{1}{2}\cdot\left[|A_{ij}|-\left(1 \text{ if ``Yes''}\right)\right]+1 \pmod{p}.
\end{align*}

$$\begin{array}{c|r|c|c|c|c|c|c}
    p & \leg{c}{p} & A_i & A_{ij} & \frac{p-1}{2}\in A_{ij} & \sum_{a\in A_{ij}}{(a^2 +a)} & |A_{ij}| & |A_i\backslash\{-1\}|\\ \hline
    1 & 1 & A_0 & A_{00} & \text{Yes} & \frac{1}{32} & \frac{p-5}{4} & \frac{p-3}{2}\\
    3 & 1 & A_0 & A_{01} & \text{Yes} &-\frac{1}{32} & \frac{p+1}{4} & \frac{p-1}{2}\\
    5 & 1 & A_0 & A_{00} & \text{No}  & \frac{1}{32} & \frac{p-5}{4} & \frac{p-3}{2}\\
    7 & 1 & A_0 & A_{01} & \text{No}  &-\frac{1}{32} & \frac{p+1}{4} & \frac{p-1}{2}\\
    1 &-1 & A_1 & A_{11} & \text{No}  & \frac{1}{32} & \frac{p-1}{4} & \frac{p-1}{2}\\
    3 &-1 & A_1 & A_{10} & \text{No}  &-\frac{1}{32} & \frac{p-3}{4} & \frac{p-3}{2}\\ 
    5 &-1 & A_1 & A_{11} & \text{Yes} & \frac{1}{32} & \frac{p-1}{4} & \frac{p-1}{2}\\
    7 &-1 & A_1 & A_{10} & \text{Yes} &-\frac{1}{32} & \frac{p-3}{4} & \frac{p-3}{2}
\end{array}$$
\end{theorem}

\begin{proof}
We will describe the columns, how they were obtained, and how they contribute to the various components of the formulas:
\begin{enumerate}
    \item The $p$ column indicates the remainder of $p$ modulo $8$.
    \item The $\leg{c}{p}$ column says whether $c$ is a quadratic residue or a quadratic non-residue modulo $p$.
    \item The $A_i$ column is $A_0$ if $\leg{c}{p}=1$ and is $A_1$ if $\leg{c}{p}=-1$. For the computation of $V$, we start with $\sum_{a\in A_i}(a^2 + a)\equiv 0$ and later subtract duplicate summands.
    \item The $A_{ij}$ column states the set in which belongs the $a\in \mathbb{Z}_p$ that resulted in non-trivial duplicate summands in $\sum_{a\in A_i}{(a^2 + a)}$. The index $i$ again corresponds to $\leg{c}{p}=(-1)^i$. The index $j$ is determined by the requirement $\leg{a+1}{p}=(-1)^j$, which is found using Lemma \ref{lem:overcounting-classification}.
    \item The $\frac{p-1}{2}\in A_{ij}$ column determines whether $a=\frac{p-1}{2}$ needs to be omitted from $A_{ij}$ when duplicates are removed from $A_i$. As explained before Lemma \ref{lem:plusminus-2-legendre}, this is because the pair $(a,a+1)$ and $(b,b+1) = (-a-1,-a)$ are actually the same modulo $p$. So, even if $a=\frac{p-1}{2}\in A_{ij}$, it has not been included twice originally and needs not be subtracted. This column was computed using Lemma \ref{lem:plusminus-2-legendre}.
    \item The $\sum_{a\in A_{ij}}{(a^2 +a)}$ column was computed in Lemma \ref{lem:subtraction-sum-computations}. The factors of $\frac{1}{2}$ in $V$ and $W$ come from the fact that we want to remove only one of each pair $(a,a+1)$ and $(b,b+1) = (-a-1,-a)$. Note that the $-\frac{1}{4}$ in $V$ comes from $\left(\frac{p-1}{2}\right)\cdot \left(\frac{p+1}{2}\right) \equiv -\frac{1}{4}.$
    \item The $|A_{ij}|$ column follows directly from Lemma \ref{lem:Lemmermeyer}.
    \item The $|A_i\backslash\{-1\}|$ column was found by considering whether $a=-1$ was already in $A_i$, using Lemma \ref{lem:1st-supplement}. If so, $1$ was subtracted from $|A_i| = \frac{p-1}{2}$ The reason that it is removed is that $a+1=0$, and we want to avoid the pairs $(a,a+1) = (-1,0)$ and $(b,b+1) = (-a-1,-a)= (0,1)$ in the initial count; $b=0$ is not in $A_i$ by definition. One of $a=0,-1$ is placed back into the computation of $W$ with the $+1$ on the far right of $W$.
\end{enumerate}
Therefore, the general idea behind both the sum computation and the cardinality computation is to find it with repetition among the summands or elements, then remove repeated ones (which can only exist in dual pairs), while considering whether the middle terms of $\mathbb{Z}_p$ form such a dual pair.
\end{proof}

\begin{corollary}\label{cor:sum-of-residues}
Theorem \ref{thm:master-theorem}, works out explicitly as follows, as predicted by the authors of \cite[p.~7]{Harrington2}.
$$
\begin{array}{c|r|r|r}
    p\pmod{8} & \leg{c}{p} & V & W \\ \hline
    1 & 1  & -\frac{9}{64} & \frac{5}{8} \\
    3 & 1  & -\frac{7}{64} & \frac{7}{8} \\
    5 & 1  & -\frac{1}{64} & \frac{1}{8} \\
    7 & 1  & \frac{1}{64}  & \frac{3}{8} \\
    1 & -1 & -\frac{1}{64} & \frac{5}{8} \\
    3 & -1 & \frac{1}{64}  & -\frac{1}{8} \\
    5 & -1 & -\frac{9}{64} & \frac{9}{8} \\
    7 & -1 & -\frac{7}{64} & \frac{3}{8}
\end{array}
$$
\end{corollary}

\begin{proof}
For illustrative purposes, we compute $V$ and $W$ in the first row. Following the formula and data in Theorem \ref{thm:master-theorem},
\begin{align*}
    V &\equiv -\frac{1}{2}\cdot \left[\sum_{a\in A_{ij}}{(a^2 + a)}-\left(-\frac{1}{4} \text{ if ``Yes''}\right)\right]\\
    &\equiv -\frac{1}{2}\cdot \left[\frac{1}{32}-\left(-\frac{1}{4}\right)\right] \equiv -\frac{9}{64}\pmod{p},\\
    W &\equiv |A_i\backslash\{-1\}| - \frac{1}{2}\cdot\left[|A_{ij}|-\left(1 \text{ if ``Yes''}\right)\right]+1\\
    &\equiv \frac{p-3}{2} - \frac{1}{2}\cdot\left(\frac{p-5}{4}-1\right)+1\equiv \frac{5}{8}\pmod{p}.
\end{align*}
The other rows follow similarly, completing the proof of the conjecture in \cite[p.~7]{Harrington2}.
\end{proof}

We acknowledge that the result extends easily to the non-monic case. For any leading coefficient $a\not\equiv 0\pmod{p}$,
$$R(ax^4 + cx^2 + e) = a\cdot R(x^4 + a^{-1}cx^2 + a^{-1}e).$$
So, beyond fixing the remainder of $p$ modulo $8$, the criteria for casework becomes whether $\leg{a^{-1}c}{p}=\leg{a}{p}\leg{c}{p}$ is $1$ or $-1$. This tacks on one copy of $a$ in the denominator of the first term of each case. For example, if $p\equiv 1\pmod{8}$ and $\leg{a}{p}\leg{c}{p} =1$, then
\begin{align*}
    \sum{R(ax^4 + cx^2 + e)} &\equiv a\cdot\sum{R(x^4 + a^{-1}cx^2 + a^{-1}e)}\\
    &\equiv a\cdot \left(\frac{-9a^{-2}c^2}{64}+\frac{5a^{-1}e}{8}\right)\equiv \frac{-9c^2}{64a}+\frac{5e}{8}.
\end{align*}

\section{The Set of Sums}

Now we will compute $S(p)$, as defined in Definition \ref{def:set-of-sum-of-residues}, in terms of $R_p (x^2)$ and $R_p (x^4)$. We start by reorganizing Corollary \ref{cor:sum-of-residues} as follows.

\begin{lemma}\label{lem:sum-summary}
For $f(x) = x^4 + cx^2 + e$, where $c$ and $e$ are integers such that $p\nmid c$, $\sum{R(f)}$ may be computed across eight cases as follows, according to the remainder of $p$ modulo $8$, and whether $\leg{c}{p} = 1$ or $-1$.

$$
\begin{array}{c|r|r}
    p\pmod{8} & \leg{c}{p} = 1 & \leg{c}{p} = -1 \\ \hline
    3 & -\frac{7}{64}c^2 +\frac{7}{8}e & \frac{1}{64}c^2 -\frac{1}{8}e \\
    7 & \frac{1}{64}c^2 +\frac{3}{8}e & -\frac{7}{64}c^2 +\frac{3}{8}e \\
    1 & -\frac{9}{64}c^2 +\frac{5}{8}e & -\frac{1}{64}c^2 +\frac{5}{8}e \\
    5 & -\frac{1}{64}c^2 +\frac{1}{8}e & -\frac{9}{64}c^2 +\frac{9}{8}e \\
\end{array}
$$
For $p\mid c$, \cite[p.~4]{Harrington2} proves that $$\sum{R(x^4)} \equiv 0 \pmod{p}.$$
\end{lemma}

We will need some results about quartic residues.

\begin{definition}
For our purposes, we define the power residue symbol for integers $a$ and $k\ge 2$ as
$$
\leg[k]{a}{p} =
\begin{cases}
    \hfill 1 &\text{ if } x^k \equiv a \pmod{p} \text{ has a non-zero solution } x\\
    \hfill 0 &\text{ if } a\equiv 0\pmod{p}\\
    \hfill -1 &\text{ otherwise}
\end{cases}.
$$

If $k=2$, we drop the subscript of $2$ from the power residue symbol $\leg[2]{a}{p}=\leg{a}{p}$, as is typical for the Legendre symbol, as defined in Definition \ref{def:Legendre-symbol}.
\end{definition}

\begin{lemma}\label{lem:-1-quartic-residue}
According to Gauss \cite[\S~3]{Gauss-biquadratic1}, if $p\equiv 3\pmod{4}$ and $a$ is an integer, then
$$
\leg[4]{a}{p} = 1 \iff \leg{a}{p} = 1.
$$
As a result of Lemma \ref{lem:1st-supplement}, if $p\equiv 3\pmod{4}$, then
$$\leg[4]{-1}{p} = -1.$$
Also according to Gauss \cite[\S~10]{Gauss-biquadratic1}, if $p\equiv 1\pmod{4}$, then there are two cases
$$
\leg[4]{-1}{p}
=
\begin{cases}
    \hfill 1 &\text{ if } p\equiv 1\pmod{8}\\
    \hfill -1 &\text{ if } p \equiv 5\pmod{8}
\end{cases}.
$$
\end{lemma}

\begin{theorem}\label{thm:3-mod-4}
Let $p\equiv 3\pmod{4}$. Then
$$
S(p) =
\begin{cases}
    \hfill \mathbb{Z}_p &\text{ if } -1\in S(p)\\
    \hfill R (x^2) &\text{ if } -1 \not\in S(p)
\end{cases}.
$$
\end{theorem}

\begin{proof}
Suppose $p\equiv 3\pmod 8$. By the first row of Lemma \ref{lem:sum-summary}, the fact that $2\in A_{1}$ from Lemma \ref{lem:2nd-supplement}, and applying Lemma \ref{lem:closure-cosets} to multiply $A_{0}$ or $A_{1}$ by $2^{-2}\in A_{0}$ or $2^{-1}\in A_{1}$,
\begin{align*}
    S(p) &= R_{c\in A_{0}}\left(-\frac{7}{64}c^2\right)\cup R_{c\in A_{1}}\left(\frac{1}{64}c^2\right)\cup \{0\}\\
    &= R_{c\in A_{0}}\left(-7(2^{-1}\cdot 2^{-2}\cdot c)^2\right)\cup R_{c\in A_{1}}\left((2^{-1}\cdot 2^{-2}\cdot c)^2\right)\cup \{0\}\\
    &= R_{c\in A_{0}}\left(-7(2^{-1}\cdot c)^2\right)\cup R_{c\in A_{1}}\left((2^{-1}\cdot c)^2\right)\cup \{0\}\\
    &= R_{c\in A_{1}}\left(-7c^2\right)\cup R_{c\in A_{0}}\left(c^2\right)\cup \{0\}.
\end{align*}

Now suppose $p\equiv 7\pmod 8$. By the second row of Lemma \ref{lem:sum-summary}, the fact that $2\in A_{0}$ by Lemma \ref{lem:2nd-supplement}, and a usage of Lemma \ref{lem:closure-cosets} that is similar to the previous case,
\begin{align*}
    S(p) &= R_{c\in A_{0}}\left(\frac{1}{64}c^2\right)\cup R_{c\in A_{1}}\left(-\frac{7}{64}c^2\right)\cup \{0\}\\
    &= R_{c\in A_{0}}\left((2^{-1}\cdot 2^{-2}\cdot c)^2\right)\cup R_{c\in A_{1}}\left(-7(2^{-1}\cdot 2^{-2}\cdot c)^2\right)\cup \{0\}\\
    &= R_{c\in A_{0}}\left((2^{-1}\cdot c)^2\right)\cup R_{c\in A_{1}}\left(-7(2^{-1}\cdot c)^2\right)\cup \{0\}\\
    &= R_{c\in A_{0}}\left(c^2\right)\cup R_{c\in A_{1}}\left(-7c^2\right)\cup \{0\}.
\end{align*}

Thus, for both $p\equiv 3,7\pmod{8}$, and so for $p\equiv 3\pmod{4}$ in general,
$$S(p) = R_{c\in A_{0}}\left(c^2\right)\cup R_{c\in A_{1}}\left(-7c^2\right)\cup \{0\}.$$
By Lemma \ref{lem:closure-cosets}, $A_{1}=-A_{0}$. Separately, since $x^4 = (x^2)^2$, the set of quartic residues is the set of squares of quadratic residues; in the case of $p\equiv 3\pmod{4}$, Lemma \ref{lem:-1-quartic-residue} says that the set of quartic residues is also the set of quadratic residues. Then the set of squares of quadratic residues is the set of quadratic residues. So we get
\begin{align*}
    S(p) &= R_{c\in A_{0}}\left(c^2\right)\cup R_{c\in A_{0}}\left(-7(-c)^2\right)\cup \{0\}\\
    &= R_{c\in A_{0}}\left(c^2\right)\cup R_{c\in A_{0}}\left(-7c^2\right)\cup \{0\}\\
    &= R_{c\in A_{0}}\left(c\right)\cup R_{c\in A_{0}}\left(-7c\right)\cup \{0\}.
\end{align*}
Recall that, since it is assumed that $p\equiv 3\pmod{4}$, Lemma \ref{lem:1st-supplement} says that $-1\in A_{1}$. As a result, in the above representation of $S(p)$, we see that $-1\not\in R_{c\in A_{0}}\left(c\right)$, so $$-1\in S(p) \iff -1 \in R_{c\in A_{0}}\left(-7c\right).$$
This biconditional statement is equivalent to $7$ having an inverse that is in $A_{0}$, which is equivalent to $7$ itself being in $A_{0}$. Therefore, using $-1\in A_{1}$:
\begin{enumerate}
    \item If $-1\in S(p)$, then $7\in A_{0}$, so
    \begin{align*}
        S(p) &= R_{c\in A_{0}}\left(c\right)\cup R_{c\in A_{0}}\left(-7c\right)\cup \{0\}\\
        &= R_{c\in A_{0}}\left(c\right)\cup R_{c\in A_{0}}\left(-c\right)\cup \{0\}\\
        &= R_{c\in A_{0}}\left(c\right)\cup R_{c\in A_{1}}\left(c\right)\cup \{0\}\\
        &= \mathbb{Z}_p.
    \end{align*}
    \item If $-1\not\in S(p)$, then $7\in A_{1}$ or $p=7$. In the $7\in A_{1}$ case,
    \begin{align*}
        S(p) &= R_{c\in A_{0}}\left(c\right)\cup R_{c\in A_{0}}\left(-7c\right)\cup \{0\}\\
        &= R_{c\in A_{0}}\left(c\right)\cup R_{c\in A_{1}}\left(-c\right)\cup \{0\}\\
        &= R_{c\in A_{0}}\left(c\right)\cup R_{c\in A_{0}}\left(c\right)\cup \{0\}\\
        &= R (x^2).
    \end{align*}
    If $p=7$, then
    \begin{align*}
        S(p) &= R_{c\in A_{0}}\left(c\right)\cup R_{c\in A_{0}}\left(-7c\right)\cup \{0\}\\
        &= R_{c\in A_{0}}\left(c\right)\cup R_{c\in A_{0}}\left(0\right)\cup \{0\}\\
        &= R_{c\in A_{0}}\left(c\right)\cup \{0\}\\
        &= R (x^2).
    \end{align*}
    Either way, if $-1\not\in S(p)$, then $S(p) = R (x^2)$.
\end{enumerate}
\end{proof}

\begin{theorem}\label{thm:1-mod-4}
Let $p\equiv 1\pmod{4}$. Then
$$
S(p) =
\begin{cases}
    \hfill R (x^2) &\text{ if } -1\in S(p)\\
    \hfill R (x^4) &\text{ if } -1\not\in S(p),\; p\equiv 5\pmod{8}\\
    \hfill \left(R (x^2)\backslash R (x^4)\right)\cup\{0\} &\text{ if } -1 \not\in S(p),\; p\equiv 1\pmod{8}
\end{cases}.
$$
\end{theorem}

\begin{proof}
Suppose $p\equiv 1\pmod 8$. By the third row of Lemma \ref{lem:sum-summary}, the fact that $2\in A_{0}$ from Lemma \ref{lem:2nd-supplement}, and using Lemma \ref{lem:closure-cosets} to multiply $A_{0}$ or $A_{1}$ by $2^{-2}\in A_{0}$ or $2^{-1}\in A_{0}$,
\begin{align*}
    S(p) &= R_{c\in A_{0}}\left(-\frac{9}{64}c^2\right)\cup R_{c\in A_{1}}\left(-\frac{1}{64}c^2\right)\cup \{0\}\\
    &= R_{c\in A_{0}}\left(-(3\cdot 2^{-1}\cdot 2^{-2}\cdot c)^2\right)\cup R_{c\in A_{1}}\left(-(2^{-1}\cdot 2^{-2}\cdot c)^2\right)\cup \{0\}\\
    &= R_{c\in A_{0}}\left(-(3\cdot 2^{-1}\cdot c)^2\right)\cup R_{c\in A_{1}}\left(-(2^{-1}\cdot c)^2\right)\cup \{0\}\\
    &= R_{c\in A_{0}}\left(-(3c)^2\right)\cup R_{c\in A_{1}}\left(-c^2\right)\cup \{0\}.
\end{align*}

Now suppose $p\equiv 5\pmod 8$. By the fourth row of Lemma \ref{lem:sum-summary}, the fact that $2\in A_1$ by Lemma \ref{lem:2nd-supplement}, and a usage of Lemma \ref{lem:closure-cosets} that is similar to the previous case,
\begin{align*}
    S(p) &= R_{c\in A_{0}}\left(-\frac{1}{64}c^2\right)\cup R_{c\in A_{1}}\left(-\frac{9}{64}c^2\right)\cup \{0\}\\
    &= R_{c\in A_{0}}\left(-(2^{-1}\cdot 2^{-2}\cdot c)^2\right)\cup R_{c\in A_{1}}\left(-(3\cdot 2^{-1}\cdot 2^{-2}\cdot c)^2\right)\cup \{0\}\\
    &= R_{c\in A_{0}}\left(-(2^{-1}\cdot c)^2\right)\cup R_{c\in A_{1}}\left(-(3\cdot 2^{-1}\cdot c)^2\right)\cup \{0\}\\
    &= R_{c\in A_{1}}\left(-c^2\right)\cup R_{c\in A_{0}}\left(-(3c)^2\right)\cup \{0\}.
\end{align*}

Thus, for both $p\equiv 1,5\pmod{8}$, and so for $p\equiv 1\pmod{4}$ in general,
$$S(p) = R_{c\in A_{0}}\left(-(3c)^2\right)\cup R_{c\in A_{1}}\left(-c^2\right)\cup \{0\}.$$

We claim that $-1\not\in R_{c\in A_{1}}\left(-c^2\right)$. This is because $$-1\in  R_{c\in A_{1}}\left(-c^2\right) \iff c^2 = 1 \iff c = \pm 1$$ for some $c\in A_{1}$. But both $1$ and $-1$ are in $A_{0}$ for $p\equiv 1\pmod{4}$, by Lemma \ref{lem:1st-supplement}. As a result, in the above representation of $S(p)$, we see that $-1\not\in  R_{c\in A_{1}}\left(-c^2\right)$, so
$$-1\in S(p) \iff -1\in R_{c\in A_{0}}\left(-(3c)^2\right).$$ This biconditional statement is equivalent to there existing a $c\in A_{0}$ such that $-(3c)^2=-1$, which is equivalent to $3c=1$ or $3c=-1$. In turn, since $-1\in A_{0}$ for $p\equiv \pmod{4}$, the possibility of $3c=-1$ can be absorbed into the possibility of $3c=1$, which is equivalent to $3\in A_{0}$. Therefore, using $-1\in A_{0}$:
\begin{enumerate}
    \item Suppose $p\equiv 1\pmod{4}$ and $-1\in S(p)$. Then $3\in A_{0}$ by above, so
    \begin{align*}
        S(p) &= R_{c\in A_{0}}\left(-(3c)^2\right)\cup R_{c\in A_{1}}\left(-c^2\right)\cup \{0\}\\
        &= R_{c\in A_{0}}\left(-c^2\right)\cup R_{c\in A_{1}}\left(-c^2\right)\cup \{0\}\\
        &= R_{c\in \mathbb{Z}_p}\left(-c^2\right)\\
        &= (-1)\cdot A_{0}' = A_{0}' = R (x^2).
    \end{align*}
    \item Suppose $p\equiv 1\pmod{4}$ and $-1\not\in S(p)$. Then $3\in A_{1}$ by above, so
    \begin{align*}
        S(p) &= R_{c\in A_{0}}\left(-(3c)^2\right)\cup R_{c\in A_{1}}\left(-c^2\right)\cup \{0\}\\
        &= R_{c\in A_{1}}\left(-c^2\right)\cup R_{c\in A_{1}}\left(-c^2\right)\cup \{0\}\\
        &= R_{c\in A_{1}}\left(-c^2\right)\cup \{0\}.
    \end{align*}
    \begin{enumerate}
        \item Suppose $p\equiv 5\pmod{8}$. Then $\leg[4]{-1}{p} = -1$, but $\leg{-1}{p} = 1$. The latter implies that $-1 \equiv \gamma^2\pmod{p}$ for some $\gamma\in \mathbb{Z}_p\backslash\{0\}$, and the former ensures that $\gamma \not \in A_{0}$, otherwise $-1$ would be the square of a quadratic residue, as in a quartic residue. So $\gamma \in A_{1}$, implying
        \begin{align*}
            S(p) &= R_{c\in A_{1}}\left(-c^2\right)\cup \{0\}\\
            &= R_{c\in A_{1}}\left(\gamma^2 c^2\right)\cup \{0\}\\
            &= R_{c\in A_{1}}\left((\gamma c)^2\right)\cup \{0\}\\
            &= R_{c\in A_{0}}\left(c^2\right)\cup \{0\}\\
            &= R (x^4).
        \end{align*}
        \item Suppose $p\equiv 1\pmod{8}$. Then $\leg[4]{-1}{p} = 1$. So $-1 \equiv \delta^4 \equiv (\delta^2)^2\pmod {p}$ for some $\delta\in \mathbb{Z}_p\backslash\{0\}$. Taking $\gamma = \delta^2$, we get that $\gamma \in A_{0}$ such that $\gamma^2 \equiv -1\pmod{p}$. Then
        \begin{align*}
            S(p) &= R_{c\in A_{1}}\left(-c^2\right)\cup \{0\}\\
            &= R_{c\in A_{1}}\left(\gamma^2 c^2\right)\cup \{0\}\\
            &= R_{c\in A_{1}}\left((\gamma c)^2\right)\cup \{0\}\\
            &= R_{c\in A_{1}}\left(c^2\right)\cup \{0\}.
        \end{align*}
        According to Gauss \cite[\S~4-7]{Gauss-biquadratic1}, $R_{c\in A_{1}}\left(c^2\right)$ is precisely the set of quadratic residues that are not quartic residues, which produces the desired representation $$S(p) = \left(R(x^2)\backslash R(x^4)\right)\cup\{0\}.$$
    \end{enumerate}
\end{enumerate}
\end{proof}

Theorem \ref{thm:3-mod-4} and Theorem \ref{thm:1-mod-4} maybe reformulated in terms of explicit congruence classes of $p$ as follows.

\begin{corollary}
For any prime $p\ge 7$,
$$
S(p) =
\begin{cases}
    \hfill \mathbb{Z}_p & \text{ if } p\equiv 3,19,27\pmod{28}\\
    \hfill R (x^2) & \text{ if } p\equiv 7,11,15,23\pmod{28}\\
    \hfill R (x^2) &\text{ if } p\equiv 1,13\pmod{24}\\
    \hfill R (x^4) &\text{ if } p\equiv 5,21\pmod{24}\\
    \hfill \left(R (x^2)\backslash R (x^4)\right)\cup\{0\} &\text{ if } p\equiv 9,17\pmod{24}
\end{cases},
$$
where $p\equiv 3\pmod{4}$ in the first two cases and $p\equiv 1\pmod{4}$ in the last three cases.
\end{corollary}

\begin{proof}
Using quadratic reciprocity, it is known that $7\in A_{0}$ if and only if $p\equiv \pm 1,\pm 3\pm 9\pmod{28}$, and that $3\in A_{0}$ if and only if $p\equiv \pm 1\pmod{12}$.
\begin{enumerate}
    \item Suppose $p\equiv 3\pmod{4}$. We know from Theorem \ref{thm:3-mod-4} and its proof that
    \begin{align*}
        S(p) &=
        \begin{cases}
            \hfill \mathbb{Z}_p &\text{ if } -1\in S(p)\\
            \hfill R (x^2) &\text{ if } -1 \not\in S(p)
        \end{cases}\\
        &= \begin{cases}
            \hfill \mathbb{Z}_p &\text{ if } 7\in A_{0} \\
            \hfill R (x^2) &\text{ if } 7 \in A_{1}\cup\{0\}
        \end{cases}.
    \end{align*}
    The residue classes of $p$ with $7\in A_{0}$ are $p\equiv \pm 1,\pm 3\pm 9\pmod{28}$, and the residue classes of $p$ with $7\in A_{1}$ are $p\equiv \pm 5,\pm 11\pm 13\pmod{28}$, and $p=7$ has $\leg{7}{p} = 0$ ($p=7$ is the only prime in the residues class of $7$ modulo $28$). Eliminating the classes that do not satisfy $p\equiv 3\pmod{4}$, we get
    $$S(p) =
    \begin{cases}
        \hfill \mathbb{Z}_p &\text{ if } p\equiv 3,19,27\pmod{28}\\
        \hfill R (x^2) &\text{ if } p\equiv 7,11,15,23\pmod{28}
    \end{cases},$$
    which indeed covers all cases modulo $28$ of $p\equiv 3\pmod{4}$.
    \item Suppose $p\equiv 1\pmod{4}$. We know from Theorem \ref{thm:1-mod-4} and its proof that
    \begin{align*}
        S(p) &=
        \begin{cases}
            \hfill R (x^2) &\text{ if } -1\in S(p)\\
            \hfill R (x^4) &\text{ if } -1\not\in S(p),\; p\equiv 5\pmod{8}\\
            \hfill \left(R (x^2)\backslash R (x^4)\right)\cup\{0\} &\text{ if } -1 \not\in S(p),\; p\equiv 1\pmod{8}
        \end{cases}\\
        &= \begin{cases}
            \hfill R (x^2) &\text{ if } 3\in A_{0}\\
            \hfill R (x^4) &\text{ if } 3\in A_{1},\; p\equiv 5\pmod{8}\\
            \hfill \left(R (x^2)\backslash R (x^4)\right)\cup\{0\} &\text{ if } 3 \in A_{1},\; p\equiv 1\pmod{8}
        \end{cases}\\
    \end{align*}
    The residue classes of $p$ with $3\in A_{0}$ are $p\equiv \pm 1\pmod{12}$, and the residue classes of $p$ with $3\in A_{1}$ are $p\equiv \pm 3, \pm 5\pmod{12}$. Eliminating the classes that do not satisfy $p\equiv 1\pmod{4}$, we get $3\in A_{0}$ if and only if $p\equiv 1\pmod{12}$, and $3\in A_{1}$ if and only if $p\equiv 5,9\pmod{12}$. But we need to do casework on primes congruent to $5$ or $1$ modulo $8$, so we scale the modulus of the primes up to $24$ to get $3\in A_{0}$ if and only if $p\equiv 1,13\pmod{24}$, and $3\in A_{1}$ if and only if $p\equiv 5,9,17,21\pmod{24}$. Here, $5$ and $21$ reduce to $5$ modulo $8$, and $9$ and $17$ reduce to $1$ modulo $8$. Therefore,
    $$S(p) =
    \begin{cases}
        \hfill R (x^2) &\text{ if } p\equiv 1,13\pmod{24}\\
        \hfill R (x^4) &\text{ if } p\equiv 5,21\pmod{24}\\
        \hfill \left(R (x^2)\backslash R (x^4)\right)\cup\{0\} &\text{ if } p\equiv 9,17\pmod{24}
    \end{cases},
    $$
    which indeed covers all cases modulo $24$ of $p\equiv 1\pmod{4}$.
\end{enumerate}
\end{proof}

\section{Excluded Primes}

For completeness, we obtain the sum of residues and the set of these sums for $p=3$ and $p=5$.

First, we compute $\sum{R_3(x^4 + cx^2 + e)}$ and $\sum{R_5(x^4 + cx^2 + e)}$.

\begin{enumerate}
    \item Modulo $3$, $A_{0}'=\{0,1\}$ and $A_{1}'=\{0,2\}$. Then
    \begin{align*}
        R_{a\in A_{0}'}(a^2 +a)&= \{0^2 + 0,1^2 +1\}=\{0,2\},\\
        R_{a\in A_{1}'}(a^2 +a)&= \{0^2 + 0,2^2 +2\}=\{0\},\\
        \sum{R_3(x^4 + cx^2 + e)} &\equiv
        \begin{cases}
            \hfill (0+2)c^2 + 2e \equiv 2c^2 + 2e &\text{ if } c\in A_{0}\\
            \hfill 0c^2 + 1e \equiv e &\text{ if } c\in A_{1}
        \end{cases}.
    \end{align*}
    \item Modulo $5$, $A_{0}' = \{0,1,4\}$ and $A_{1}' = \{0,2,3\}$. Then
    \begin{align*}
        R_{a\in A_{0}'}(a^2 +a)&= \{0^2 + 0,1^2 +1,4^2+4\}=\{0,2\},\\
        R_{a\in A_{1}'}(a^2 +a)&= \{0^2 + 0,2^2 +2, 3^2 + 3\}=\{0,1,2\},\\
        \sum{R_5(x^4 + cx^2 + e)} &\equiv
        \begin{cases}
            \hfill (0+2)c^2 + 2e \equiv 2c^2 + 2e &\text{ if } c\in A_{0}\\
            \hfill (0+1+2)c^2 + 3e \equiv 3c^2 +3e &\text{ if } c\in A_{1}
        \end{cases}.
    \end{align*}
\end{enumerate}

Secondly, we compute $S(3)$ and $S(5)$.
\begin{enumerate}
    \item By above, modulo $3$,
    \begin{align*}
        A_{0}' &=\{0,1\},\\
        A_{1} &=\{2\},\\
        \sum{R_3 (x^4 + cx^2)} &\equiv
        \begin{cases}
            \hfill 2c^2 \pmod{3} &\text{ if } c\in A_0\\
            \hfill 0  \pmod{3} &\text{ if } c\in A_1
        \end{cases}.
    \end{align*}
    As a result, $$S(3)=\{0, 2\cdot 1^2,0\}=\{0,2\}$$
    \item Again, by above, modulo $5$,
    \begin{align*}
        A_{0}' &=\{0,1,4\},\\
        A_{1} &=\{2,3\},\\
        \sum{R_5 (x^4 + cx^2)} &\equiv
        \begin{cases}
            \hfill 2c^2 \pmod{5} &\text{ if } c\in A_0\\
            \hfill 3c^2 \pmod{5} &\text{ if } c\in A_1
        \end{cases}.
    \end{align*}
    As a result, $$S(5)=\{0, 2\cdot 1^2,3\cdot 2^2,3\cdot 3^2,2\cdot 4^2\}=\{0,2\}.$$
\end{enumerate}

It does not seem to be possible to place these computations for $p=3,5$ under the same arguments or formulations as for primes $p\ge 7$.

\end{document}